\documentclass[11pt]{amsart}
\usepackage{enumerate}
\usepackage[all]{xy}
\usepackage{amsmath}
\usepackage{amsfonts}
\usepackage{amssymb}
\SelectTips{xy}{9}
\theoremstyle{plain}
\newtheorem{theorem}{Theorem}[section]
\newtheorem{prop}[theorem]{Proposition}
\newtheorem{cor}[theorem]{Corollary}

\def\det{\mathop{\mathrm{det}}\nolimits}
\def\Im{\mathop{\mathrm{Im}}\nolimits}
\def\Ker{\mathop{\mathrm{Ker}}\nolimits}
\def\Coker{\mathop{\mathrm{Coker}}\nolimits}

\def\Spec{\mathop{\mathrm{Spec}}\nolimits}

\def\CH{\mathop{\mathrm{CH}}\nolimits}

\newcommand{\bb}[1]{{\mathbb{#1}}}

\def\Ker{\mathop{\mathrm{Ker}}\nolimits}

\def\Coker{\mathop{\mathrm{Coker}}\nolimits}

\def\Spec{\mathop{\mathrm{Spec}}\nolimits}

\def\Z/2{\mathbb{F}_2}
\def\Z{\mathbb{Z}}
\def\C{\mathbb{C}}
\newdir{ >}{{}*!/-7pt/@{>}}

\begin{document}
\title{The motivic cohomology of $BSO_n$}
%

\title{The motivic cohomology of $BSO_{n}$}
\author{Masana Harada}
\address{Department of Mathematics,
Kyoto University, Kyoto 606-8502, Japan}
\email{harada@kusm.kyoto-u.ac.jp}
\author{Masayuki Nakada}
\address{Kobe University~Secondary~School.Kyoto University}
\email{mnakada@people.kobe-u.ac.jp}
\maketitle
\begin{abstract}
We will determine the motivic cohomology 
$H^{\ast , \ast}(BSO_n ,\, \Z/2)$ with coefficients in $\Z/ 2$  of
the classifying space  of special orthogonal groups $SO_n$
over the complex numbers $\bb{C}$.
\end{abstract}

\setcounter{section}{-1}
\section{Introduction}
We denote by $BG_k$ the classifying space of 
a linear algebraic group $G_k$ over a field $k$.
It is a colimit of smooth schemes defined  by B.~Totaro~\cite{tot1999},
and F.~Morel and V.~Voevodsky in~\cite[\S 4]{m-v1999},
where it is called the geometric model.
$SO_n$ is the special orthogonal group of dimension $n$,
and we will treat cases where $k$ is 
the field of complex numbers $\bb{C}$.
By $H^{\ast , \ast}(X)$  we mean the bi-graded 
motivic cohomology group with coefficients in $ \Z/ 2$,
$$
\displaystyle\operatornamewithlimits\oplus_{i,j \geq 0} H^{i ,j}(X)=
\displaystyle\operatornamewithlimits\oplus_{i,j \geq 0} H^{i}(X,\, \Z/2(j)),$$ of 
a smooth scheme $X$.
We use a definition used in~\cite{voe2003-2} and~\cite{voe2003}.
We call the first index $i$ by degree and $2j-i$ by weight.
(Total weight may be appropriate. But we simply call this way.) 
Since we will only consider cohomology with  coefficients in $\Z/ 2$,
the coefficients are usually omitted in the notation.
The motivic cohomology of $BSO_n$ is
defined as a limit of motivic cohomology of smooth schemes.

The motivic cohomology groups are related to the 
Chow groups and 
ordinary cohomology groups in the following way.
$H^{2i, i}(X )$ is isomorphic to the
mod $2$ reduction $CH^i(X) \otimes \Z/ 2$ of 
the $i$-th Chow group for a smooth scheme $X$.
In our convention
$\oplus H^{2i, i}(X )$ is the subgroup of elements having weight $0$.
There are no elements having negative weights.
 There exists
a realization mapping
\begin{equation}
t^{i,j}_X : H^{i,j}(X , \, \Z/ 2)\to H^{i}(X ,\, \Z/ 2),
\end{equation}
where
$t^{2i,i}_X$ is identified to the cycle class mapping.
It was called the Beilinson-Lichtenbaum conjecture~\cite[\S 6]{voe2003-2} that
it becomes an isomorphism when $i \leq j$,
which is proved  in~\cite{voe2003-2}.

$H^{\ast}(BG, \, \Z/ 2)=H^{\ast}(BG)$
is  the
singular cohomology group 
of CW-complex $BG(\C)$, which is homotopy equivalent 
to the classifying space of a maximal compact subgroup of $G(\C)$.

When $G= GL_n, SL_n, Sp_n$, the motivic cohomology of $BG$ is known 
over arbi\-trary fields integrally. 
For example
$
H^{\ast , \ast}(BGL_n ,\,\Z )$
is equal to 
$
H^{\ast , \ast}( \Spec k,\,\Z)
[\, c_1 , \ldots , c_n]
$
where $c_i$ is the Chern class of the standard representation.
N.~Yagita determined $H^{\ast,\ast}(BSO_4)$ over $\mathbb{C}$~\cite{yag2010}, 
where
cases for 
$G=O_n,\,SO_{2m+1}, \, G_2, \,Spin_7$ 
and some other groups in different coefficients are treated.
It completes his preceding results in~\cite{yag2003}, ~\cite{yag2005}.
We will recall his argument in 1.3.
Note that $SO_2=GL_1$ for $\sqrt{-1}\in k$.
Therefore the new 
results in this paper are cases $G=SO_{2m}$ for $m \geq 3$.
 
We know that 
D.~Edidin and W.~Graham~\cite{e-g1995}
constructed 
an element 
in $CH^m(SO_{2m})$ for each $m \geq 2$
such that the mod 2 reduction of it,
$y_{0,m}$ in $H^{2m,m}(BSO_{2m})$,
becomes 0
through cycle class mapping
$$t^{\ast,\ast}_{BSO_{2m}}(y_{0,m})=0.$$
(The entire structure of the group $CH^{\ast}(SO_{2m})$ is known by
R.~Pandharipande ~\cite{pand1998} in the case $m=2$
and R.~Field~\cite{fie2001a}~\cite{fie2001} in general $m$.)
L.~Molina-Rojas and A.~Vistoli~\cite{m-v2006} gave a new insight on 
the element using the stratified method.
We will review it in 1.2.
This phenomenon is first observed by B.Totaro~\cite{tot1997}, 
that a generator in cobordism groups of $BSO_4(\C)$ studied in~\cite{k-y1993}
would not map to a generator of ordinary cohomology by
his refined cycle class mapping.
Note that the cobordism group of $BSO_n(\C)$
is determined in ~\cite{i-y2010}.

N.~Yagita found a series of elements $y_{i,m}$ 
in $H^{2m, i+m}(BSO_{2m})$ for any 
$i$ from $0$ to $m-2$ in the the kernel of 
$t^{\ast, \ast}_{BSO_{2m}}$~\cite[9.3]{yag2010},
where it was denoted by $u_{i+1}$.
These elements will be defined in 3.3.
Since $H^{\ast, \ast}(\Spec \bb{C})=\Z/2[\tau]$
($\tau \in H^{0, 1}(\Spec \bb{C})$)~\cite[7.8]{voe2003-2},
$H^{\ast, \ast}(BSO_n)$ is a $\Z/2[\tau]$-algebra.
$\Ker(t^{\ast,\ast}_{BSO_{2m}})$ becomes  the ideal of $\tau$-torsion elements
exactly~\cite[\S 6]{s-v2000}.

Our result is that $\Ker t^{\ast, \ast}_{BSO_{2m}}$ is
generated by 
 Yagita's elements $y_{i,m}$ 
over the polynomial ring of Chern classes of even indices; 
$0=\tau y_{i,m}=z y_{i,m}$
if $z$ is not a polynomial of $c_{2i}$'s and $y_{i,m}y_{j,m}=0$:
\begin{theorem}
The kernel
of realization mapping $\Ker t^{\ast,\ast}$
in $H^{\ast,\ast}(BSO_{2m})$ is
$$\Ker(t^{\ast,\ast}_{BSO_{2m}})=\Z/2[c_2, c_4, \ldots ,c_{2m}]\{y_{0,m}\, , \ldots,\, y_{m-2,m} \}.$$
\end{theorem}

Note that $\Ker(t^{2m,2m-1}_{X})=0$ for any $X$ 
(It is a direct consequence of the theorems of Voevodsky \cite[Theorem 6.1, Lemma 6.9, Theorem 7.4]{voe2003-2}).
%

\subsection{Conventions.}
Since we only consider coefficients in $\Z/2$, we will omit it.
We have $H^{\ast , \ast}(\Spec \bb{C} )= \Z/2[ \tau]$
where $\tau$ has degree $0$ and weight $2$.
The motivic cohomology groups $H^{\ast , \ast}(X)$
is a bi-graded algebra 
over $\Z/2[ \tau]$.

The group of quadratic roots of unity $\mu_2$ has 
the motivic cohomology 
$H^{\ast , \ast}(B\mu_2 )= \Z/2[ \tau ,y ]\otimes \Lambda(x)
=\Z/2[ \tau ,y ]\{ 1, x\}$
where $x \in H^{1 , 1}(B\mu_2 )$ and $x^2=\tau y$~\cite[6.10]{voe2003}.
Using the Beilinson-Lichtenbaum conjecture, we will identify
$H^{i}(X)= H^{i,i}(X)$.
We use the equivariant motivic cohomology in 1.1,
which is denoted by $H^{\ast,\ast}_{G}(X)$.
We use a simplified notation $H^{\ast,\ast}_{G}=H^{\ast,\ast}_{G}(\Spec \bb{C})=
H^{\ast,\ast}(BG)$.
Similarly 
$H^{\ast}_{G}=H^{\ast}_{G}(pt)$.

The orthogonal group $O_n$ is the algebraic subgroup of $GL_n$ 
consisting of elements preserving
the quadratic from $q(x)=\sum x_i^2$.
The special orthogonal group $SO_n$ is the subgroup of $O_n$
consisting of elements with determinant one,
and the inclusion is denoted $\epsilon : SO_{n} \subset O_n$.
We regard $O_{n-1}$ as the subgroup fixing the last coordinate in $O_n$.
Similarly for $\iota : SO_{n-1} \subset SO_{n}$.
We denote by $\iota$ any of those canonical inclusions.
We use the same notation for an induced morphism between classifying spaces
as $\iota : BSO_{n-1} \subset BSO_{n}$.
$\det : O_n \to \mu_2 \subset \bb{G}_m$ will be the  determinant morphism.

Let us denote by 
$\gamma_n$ the standard real $n$-dimensional representation of
$O_n(\mathbb{R})$, or its associated bundle on $BO_n(\mathbb{R})$.
For $G=O_n, SO_n$, $w_i \in
H^i(BG(\mathbb{R}))
=H^i(BG(\mathbb{C}))=
H^i(BG)= H^{i,i}(BG)$ 
will mean the $i$-th Stiefel-Whitney class of $\gamma_n$ or
$\epsilon^{-1} \gamma_n$.
It is classical that
$H^{\ast}(BO_n)= \Z/2[w_1, 
\\w_2, \ldots, w_n]$
and
$H^{\ast}(BSO_n)= \Z/2[w_2, \ldots, w_n]$.
The $i$-th 
Chern class $c_i$ of $\gamma \otimes \bb{C}$ or $\epsilon^{-1}$-image of it, 
is
an element in $CH^i(BG)\otimes \Z/2 = H^{2i,i}(BG)$ 
so that
$(w_i)^2=\tau^i c_i$.

\subsection{Acknowledgments.}
The authors would like to thank Nobuaki Yagita for several discussions,
and to the referees for many useful suggestions.

\section{Preliminaries}
\setcounter{subsection}{-1}
\subsection{Motivic cohomology and the
filtration by weight.}
The realization mapping
$t=t_X=t^{i,j}_X : H^{i,j}(X)\to H^{i,i}(X)=H^{i}(X)$
is just a multiplication of $\tau^{i-j}$~\cite[\S 6]{s-v2000}.
We can introduce 
a filtration by 
weight on $H^{i}(X)$,
so that
$F^j H^{i}(X)$ consists of 
elements which are the realization 
of weight less than $j$.
It means 
$x$ has weight exactly $2j-i$ if 
there exists $x'$ as $x=  \tau^{i-j}x'$
and $j$ is the minimum of such numbers.
We denote such $x'$ as $\overline{x}$, which is not  determined
uniquely if there exists non-trivial element of $\tau$-torsion.
But in all the cases appear in our paper
we have a unique choice of such element for a homogeneous $x$
by reasoning on bi-degree.
It defines a filtered ring structure on $H^{i}(X)$~\cite[\S 2, \S 6, \S 7]{yag2010},
which is equivalent to the 
coniveau filtration of Grothendieck, after shifts on degrees and 
indices of filtration. 
The precise argument is in ~\cite[\S 7]{yag2010}.

A morphism $f : X \to Y$ induces a pullback 
$f^* : H^{\ast,\ast}(Y)\to H^{\ast,\ast}(X)$ that is a homomorphism
of $\Z/2[\tau]$-algebras;
and a filtered ring homomorphism
$t(f^* ): H^{\ast}(Y)\to H^{\ast}(X)$.

There are motivic Milnor operations $Q_k$ on the motivic 
cohomology groups~\cite[13.4]{voe1996}:
$$
Q_k : H^{i,j}(X) \to 
H^{i +  2^{k+1}-1, j +2^{k}-1}(X).
$$
(Note that we assume
$\sqrt{-1}\in k$ so that $\rho$ there is trivial.)
$Q_i$ are derivations on the motivic cohomology groups,
$Q_iQ_j=Q_jQ_i$ if $i\neq j$ and $Q_iQ_i=0$. 
They commute with pullbacks and Gysin morphisms.
If $x$ has weight $i>0$, $Q_j x$ has weight $i-1$ if it is not zero.
These are compatible with the classical Milnor operations $Q_i$ on 
$H^*(X)$ in the sense that 
$t(Q_i x) = Q_i t(\tau^{i-j} x)$.
(~\cite[3.9]{voe1996}). 
We denote it by the same notation $Q_i$.

\subsection{Equivariant motivic cohomology}
Edidin and Graham extended 
the construction of the Chow ring of $BG$ to
the equivariant Chow ring of a scheme with G-action, which can be generalized 
to motivic cohomology~\cite{e-g1998}.
Let us recall the equivariant motivic cohomology~\cite[\S 8]{yag2010}.
Suppose that $G$ acts on a  smooth quasi-affine scheme $X$. 
For each $m\geq 0$, we can choose a representation $V$ of $G$ 
with an open subscheme $E$ on which $G$ acts freely, and 
$\mathrm{codim}_V(V- E)>m$.
The quotient $(E\times X)/G$ exists as a smooth scheme.
We set
\begin{equation}
H^{m,\ast}_G(X)=H^{m,\ast}((E\times X)/G).\nonumber
\end{equation}
It is independent of the choice of such $V$ for fixed $m$ and $*$.
For a morphism of groups $\phi : H \to G$, a $G$-schemes $X$, an
$H$-scheme $Y$, and 
an $H$-equivariant morphism $ f : Y \to X$ where $H$ acts on $X$ through $\phi$,
there is a pull-back
$f^* : H^{\ast, \ast}_G(X) \to H^{\ast,\ast}_H(Y)$.
For a subgroup $H$ of $G$, we see that~\cite[2.1]{m-v2006}
\begin{equation}
H_G^{\ast,\ast}((X\times G)/H)= H_H^{\ast,\ast}(X).
\end{equation}

Let  $U$ be  a unipotent group and consider a $U$-torsor $X \to Y$, we have an isomorphism
$H^{\ast,\ast}(Y)= H^{\ast,\ast}(X)$, for $U$ is isomorphic to an affine space
and any $U$-torsor is locally trivial  
~\cite[Prop.1]{gro1958}.
In particular for a semi-direct product $G \ltimes U$ the projection onto
the component $G \ltimes U \to G$ induces an isomorphism 
$H^{\ast,\ast}( B(G\ltimes U))= H^{\ast,\ast}(BG)$~\cite[2.3]{m-v2006}.

Take any two elements $u_0,u_1 $ of $U(\C)$.
$G \to G \times U$ by $x \mapsto (x,\, u_i)$  induce the same morphism 
on $H^{\ast,\ast}(B(G\times U))\to H^{\ast,\ast}(BG)$ for both are the section of the projection.
Similarly for the multiplicative group $\bb{G}_m$: $H^{\ast,\ast}(B(G\times \bb{G}_m))= H^{\ast,\ast}(BG)[y] \to H^{\ast,\ast}(BG)$.
Let $G$ be connected.
For any element $g \in G(\C)$ the conjugation by $g$ defines a morphism
$c_g : BG \to BG$, but we know that $g$ is contained in a Borel subgroup, 
and 
$c_g^*=c_e^*$ on $H^{\ast,\ast}(BG)$ for $e$ is the unit of $G$
~\cite[Lemma.1 of Proposition 4]{gro1958}.

\subsection{The stratified method}
The equivariant motivic cohomology theory has the following 
{\it localization sequence}; if we write an equivariant closed immersion 
$i\colon Y\hookrightarrow X$ 
of codimension $s$ and open immersion $j\colon X- Y\hookrightarrow X$, 
there is a long exact sequence
\begin{equation}
\rightarrow H_G^{\ast-1,\ast}(X- Y)\xrightarrow{\delta} 
H_G^{\ast-2s,\ast-s}(Y)\xrightarrow{i_*} H_G^{\ast,\ast}(X)
\xrightarrow{j^*} H_G^{\ast,\ast}(X- Y)\xrightarrow{\delta}\  
\nonumber
\end{equation}
where $i_*$ is the Gysin mapping.

If we know a stratification of a $G$-scheme $X$ by $G$-subschemes,
we may calculate $H_G^{\ast,\ast}(X)$ by repeated use of the localization sequence above 
and (1.1).
This method is introduced by Vezzosi to compute $\CH^*(BPGL_3)$~\cite{vez2000}.
Molina-Rojas and Vistoli~\cite{m-v2006}
calculated the Chow ring of $BG$ using it in cases $G=GL_n$, $SL_n$
 and $Sp_n$ over arbitrary fields, and $G=O_n$ and $SO_n$.
 over fields of characteristic different from $2$.
The following stratification associated with $O_n$ or $SO_n$ are
used 
to calculate the Chow groups and motivic cohomology in ~\cite{m-v2006} 
and~\cite[\S 9]{yag2010}.
Consider the standard representation of $SO_n$.
It makes $\mathbb{A}^n$  an $SO_n$-scheme.
We set $B=\{x\in \bb{A}^n\smallsetminus \{0\}\mid q(x)\neq 0\}$ and $C=\{x\in \bb{A}^n\smallsetminus\{ 0\}\mid q(x)=0\}$
where
$q(x)=\sum_{k=1}^nx_k^2$.
There is a stratification
$\{0\} \subset \{0\} \cup C  \subset \mathbb{A}^n$, and $\mathbb{A}^n
\smallsetminus (\{0\} \cup C) =B$.

Then we get two localization sequences:
\begin{equation}\label{eq:locseq1}
\rightarrow H_{SO_n}^{\ast-2n, \ast-n}(\{0\})
\xrightarrow{{s}_*}H_{SO_n}^{\ast,\ast}(\bb{A}^n)
\xrightarrow{t^*} H_{SO_n}^{\ast, \ast}(\bb{A}^n\smallsetminus \{0\})
\xrightarrow{}
\end{equation}
and
\begin{equation}\label{eq:locseq2}
\rightarrow H_{SO_n}^{\ast-2, \ast-1}(C)
\xrightarrow{{i}_*} H_{SO_n}^{\ast,\ast}
(\bb{A}^n\smallsetminus\{0\})\xrightarrow{j^*} H_{SO_n}^{\ast,\ast}(B)\xrightarrow{\delta}.
\end{equation}
We have $H_{SO_n}^{\ast, \ast}(\{0\})
= H_{SO_n}^{\ast, \ast}(\bb{A}^n)$ and 
${s}_*$ is the multiplication of the $n$-th Chern class 
$c_n\in H^{2n,n}_{SO_n}(\{0\})= H^{2n,n}(BSO_n)$
by the self-intersection formula~\cite{soulynin}.
Since in the end of proof  in \ref{subsec:cn}  we will see that 
$c_n \cup \ $ is injective,  we will denote 
\begin{equation}\label{eq:identify1}
H_{SO_n}^{\ast,\ast}
(\bb{A}^n\smallsetminus\{0\})=H_{SO_n}^{\ast,\ast}/c_n.
\end{equation}

We can define an action of $\bb{G}_m\times SO_n$ on $B$ 
such that $(t,g)\in \bb{G}_m\times SO_n$ sends $v\in B$ to 
$tg(v)$, that is, the matrix multiplication of scalar $t\in \bb{G}_m$ and $g(v)\in B$.
This action is transitive and the stabilizer of $e_1$ is a subgroup 
isomorphic to $O_{n-1}$.
The inclusion 
$\kappa_1 : O_{n-1} \to \mathbb{G}_m \times SO_n$
is given by $g\mapsto \left ( \det g ,  \kappa \right)$ where
$\kappa : O_{n-1} \to SO_n$ is 
\begin{equation}\label{eq:kappa}
g \mapsto \begin{pmatrix}\det(g)&0\\0&g\end{pmatrix}.
\end{equation}
We have
$
B\cong \bb{G}_m\times SO_n/O_{n-1}$.

Choose $E$ as in 1.1.
Then we obtain the following isomorphisms
\begin{align*}
H_{SO_n}^{\ast,\ast}(B) &= 
H^{\ast,\ast}\left(
\frac{E\times (\bb{G}_m\times SO_n/O_{n-1})}{SO_n}\right)\\
=
&H^{\ast,\ast}\left(\frac{E \times \bb{G}_m}{O_{n-1}}\right)
= H_{O_{n-1}}^{\ast, \ast}(\bb{G}_m).
\end{align*}

$\bb{G}_m$ is an $O_{n-1}$-scheme via $\det$, and 
$c_1(\gamma_n \otimes \bb{C})=c_1(\det (\gamma_n \otimes \bb{C}))$,
we get the identification similar to \eqref{eq:identify1}
\begin{equation}\label{eq:B}
H_{O_{n-1}}^{\ast,\ast}(\bb{G}_m)
= H_{O_{n-1}}^{\ast, \ast}/c_1
\end{equation}
for we will  see  in the next subsection that $c_1 \cup \ $ is a monomorphism.
Note that these are $H^{\ast, \ast}_{SO_n}$-module isomorphism
through $\kappa^*$.
$j^*$ in \eqref{eq:locseq2} is identified with the induced morphism 
of $\kappa^*$.

On the other hand, we know that $SO_n$ transitively acts on $C$ and
 the stabilizer of $e_{1}+\sqrt{-1}e_2$ is isomorphic to 
 a semi-direct product of the stabilizer of
a pair $e_{1}+\sqrt{-1}e_2, e_{1}-\sqrt{-1}e_2$ and a unipotent subgroup,
so that it is isomorphic to $SO_{n-2}\ltimes \bb{A}^{n-1} $~\cite[5.2]{m-v2006}.
Thus we get that $C\cong SO_n/(SO_{n-2}\ltimes \bb{A}^{n-1})$ and this deduces the following isomorphisms~\cite[2.3]{m-v2006}:
\begin{equation*}
H_{SO_n}^{\ast,\ast}(C)= H_{SO_{n-2}\ltimes \bb{A}^{n-1}}^{\ast,\ast}= H_{SO_{n-2}}^{\ast,\ast}.
\end{equation*}

From the above calculations we can rewrite \eqref{eq:locseq2} as follows:
\begin{equation}\label{eq:locseq2-1}
\rightarrow H_{SO_{n-2}}^{\ast-2,\ast-1}
\xrightarrow{{i}_*} H_{SO_n}^{\ast,\ast}/c_n
\xrightarrow{\kappa^*} H_{O_{n-1}}^{\ast,\ast}/c_1\xrightarrow{\delta}.
\end{equation}

%
We can regard the same stratification as that of $O_n$-spaces.
The induced long exact sequence is exactly the one used by Yagita.
We can connect two long exact sequences by the change of groups
homomorphism on $\epsilon : SO_n \to O_n$:
\begin{equation}\label{eq:locseq2-3}
\xymatrix@C-15pt{
 \ar[r] & H^{\ast-2,\ast-1}_{O_{n-2}} \ar[d]^{\epsilon^*}\ar[r]^-{{(i_0)}_*}&
H^{\ast,\ast}_{O_{n}}/c_{n} \ar[r]^-{j_0^*} \ar[d]^{\epsilon^*}
& H^{\ast, \ast}_{\mu_2 \times O_{n-1}} (\mathbb{G}_m)\ar[d]^{\epsilon^*}
\ar[r] & H^{\ast-1,\ast-1}_{O_{n-2}} \ar[r] \ar[d]^{\epsilon^*}& \\
 \ar[r] & H^{\ast-2,\ast-1}_{SO_{n-2}} \ar[r]^-{{i}_*}&
H^{\ast,\ast}_{SO_{n}}/c_n \ar[r]^-{\kappa^*}& H^{\ast ,\ast}_{O_{n-1}}(\bb{G}_m) 
\ar[r]^-{\delta} & H^{\ast-2,\ast-1}_{SO_{n-2}} \ar[r] \ar[r] & \;. }
\end{equation}

The standard inclusion $\mathbb{A}^{n} \to \mathbb{A}^{n+1}$
into the first $n$ components
induces the following commutative  long exact sequences 
(Gysin mappings and pullbacks commute in this case  ~\cite[2.3]{voe2003})
\begin{equation}\label{eq:locseq2-4}
\xymatrix@C-15pt{
 \ar[r] & H^{\ast-2,\ast-1}_{SO_{n-1}} \ar[d]^{\iota^*}\ar[r]^-{i_*}&
H^{\ast,\ast}_{SO_{n+1}}/c_{n+1} \ar[r]^-{\kappa^*} \ar[d]^{\iota^*}
& H^{\ast, \ast}_{O_{n}}/c_1\ar[d]^{\iota^*}
\ar[r]^{\delta} & H^{\ast-1,\ast-1}_{SO_{n-1}} \ar[r]\ar[d]^{\iota^*}& \\
 \ar[r] & H^{\ast-2,\ast-1}_{SO_{n-2}} \ar[r]^-{{i}_*}&
H^{\ast,\ast}_{SO_{n}}/c_n \ar[r]^-{\kappa^*}& H^{\ast ,\ast}_{O_{n-1}}/c_1
\ar[r]^-{\delta} & H^{\ast-1,\ast-1}_{SO_{n-2}} \ar[r] & \; .}
\end{equation}

Later we use another inclusion
$\mathbb{A}^{n} \to \mathbb{A}^{n+1}$
into the second to $n+1$-th components.
It is conjugate to the composition of the standard inclusions by an element of $SO_{n+1}(\bb{C})$,
so that the induced morphisms in motivic cohomology  are the same (1.1).

\subsection{$H^{*,*}(BO_n)$}
It is well-known that the map
induced from the restriction to the diagonal subgroup
\begin{equation*}
H^*(BO_n)\rightarrow H^*((B\mu_2)^n)^{\Sigma_n}=\Z/2[x_1,\ldots, x_n]^{\Sigma_n}
\end{equation*}
is an isomorphism, where $\Sigma_n$ is the $n$-th symmetric group
acting by permutations on $x_i$'s.
We need three basis of the symmetric polynomials.
Recall that for any partition $\lambda=[i_1, i_2, \ldots ,i_k,0, \ldots,0 ]
=[i_1, i_2, \ldots ,i_k,  0^{n-k} ]$ of $n$
where $\sum i_j = n$ and $i_{j} \neq 0$, $k$ is called the length of $\lambda$.
For each partition $\lambda$
we can define a symmetric polynomial $m[\lambda]=\sum x_1^{i_1}x_2^{i_2}\ldots x_k^{i_k}$
where the sum is taken on the orbit of $\Sigma_n$.
These make the monomial basis of symmetric polynomials.

The $l$-th Stiefel-Whitney class $w_l$ corresponds to
the elementary symmetric polynomial of degree $l$ :
$w_l=m[1,1,\ldots,1]=m[1^l, 0^{n-l}]$, whose monomials
gives a basis of
$
H^*(BO_n)= \bb{Z}/2[w_1,\ldots,w_n].
$
 
Finally we need a basis suitable to see action of $Q$'s.
It is due to W.~S.~Wilson.
We denote by $Q(i)M$ the $\Z/2$-module generated by 
$Q_{k_1}\ldots Q_{k_j} m$ for all $0\leq k_1 < \ldots <k_j \leq i$ and $m \in M$.
Then \begin{equation}\label{eq:wildecomp}
H^*(BO_n)=\bigoplus_{k\geq -1}Q(k)G_k,
\end{equation}
here $G_{k-1}$ is
the vector space spanned by  monomial bases $m[{\lambda}]$ where
$\lambda= [2s_1+1, 2s_2+1, \ldots, 2s_k+1, {2t_1}, {2t_2}, \ldots ,{2t_q}]$
for some $k+q\leq n \; (0\leq k, q \leq n)$ and
$0\leq s_1\leq s_2\leq\ldots\leq s_k,\ 0< t_1\leq t_2\leq\ldots\leq t_q$ such that
if the number of $t$ equal to $t_u$ is odd, then there exists some $v\leq k$ such that $2s_v+2^v<2t_u<2s_v+2^{v+1}$
(\cite[2.1]{wil1984},
while the details of the combinatorial condition here 
are not necessary in our proof.).

Recall $Q_k x_i^{2j}=0$ and $Q_k x_i^{2j+1}=x_i^{2j+2k}$ in $H^*((B\mu_2)^n)$.
Any element of $H^*(BO_n)$ is a sum of
$Q_{i_1}\ldots Q_{i_s} m_{\lambda}$ for some such $\lambda$ and $0 \leq i_1 < \ldots < i_s \leq k$.
Note $Q_{i_1}\ldots Q_{i_s} m[{\lambda}]=m[2s_1+2i_1, 2s_2+2i_2, \ldots, 2s_k+2i_s, {2t_1}, {2t_2}, \ldots ,{2t_q}]$,
which is not zero by the condition on the partition.
For example $w_l$ is an element of the basis
and
$Q_0 w_l =m[2,1^{l-1}],\,Q_0 Q_0 w_l =m[2,2,1^{l-2}]+m[2,2,1^{l-2}] = 0,\,
Q_0Q_1 w_l =m[2,4, 1^{l-2}]$.

Yagita determined the 
motivic cohomology $H^{\ast , \ast}(BO_n)$.
He proved inductively that $t_{BO_n}$ 
is injective  by the stratified method~\cite[8.2]{yag2010}.
Then the description of  the filtration by weight on 
$H^{\ast}(BO_n)$ directly connected  to 
$H^{\ast , \ast}(BO_n)$ through
Milnor operations (1.5).
The result is the following:
let $(\mu_2)^n$ be the diagonal subgroup of $O_n$.
Then the restriction induces an isomorphism~\cite[8.1]{yag2010}
\begin{equation}
H^{\ast , \ast}(BO_n)=(\otimes^n H^{\ast , \ast}(B \mu_2))^{\Sigma_n}
=\Z/2[\tau]\otimes(\oplus Q(k)G_k)
\end{equation}
where $G_k$ is 
spanned by the monomial 
bases
exactly the same conditions in  \eqref{eq:wildecomp} but $x_i \in H^{1,1}$.

\begin{theorem}[{Yagita~\cite[8.1]{yag2010}}]\label{prop:yag}
$x\in H^*(BO_n)$ has weight $k$ if and only if $k$ is the maximal number such that $Q_{i_1}Q_{i_2}\ldots Q_{i_k}x\neq 0$ for some tuple 
$(i_1,i_2,\ldots,i_k)$.

In particular $m_{\lambda}$ where
$\lambda=[2s_1+1 , \ldots , 2s_{k}+1, {2t_1}, \ldots ,{2t_q}]$
has weight $k$, and
$x$ has weight $0$ if and only if $Q_ix=0$ for every $i$.
\end{theorem}

\begin{cor}\label{cor:onwt}
The Stiefel-Whitney class $w_l$ in $H^*(BO_n)$ has weight 
$l$ for $l=1,\ldots, n$.
\end{cor}

Since $\iota^*$ discards
monomials of length exactly $n$,
the following corollary is obvious.

\begin{cor}
The mapping induced from the canonical restriction
$\iota^* : H^{*,*}(BO_n) \to H^{*,*}(BO_{n-1})$ is an
epimorphism.
\end{cor}

\section{$H^{\ast}(BSO_{n})$ and $H^{*,*}(BSO_{2m+1})$}
To determine $H^{\ast,\ast}(BSO_{n})$ we can not use 
the restriction to the diagonal subgroup. 
The induced mapping on motivic cohomology is neither injective nor surjective.
We will use the exact sequences \eqref{eq:locseq1} and \eqref{eq:locseq2-1}.
For effective use of them we need to determine the coniveau filtration
on $H^{\ast}(BSO_{n})$.



\subsection{Determination of weight}
We will write down $\kappa^*$ in terms of Stiefel-Whitney classes.
Recall $\gamma_n$ is the canonical $n$-dimensional real
vector bundle over $BO_n(\mathbb{R})$.
We have the following isomorphism of vector bundles:
\begin{equation*}
\kappa^{-1}\epsilon^{-1} \gamma_n =\det(\gamma_{n-1})\oplus\gamma_{n-1}
\end{equation*}
Thus the total Stiefel-Whitney class of $\kappa^{-1}\epsilon^{-1} \gamma_n $ is
\begin{align}
&1+w_1(\kappa^{-1}\epsilon^{-1} \gamma_n )+w_2(\kappa^{-1}\epsilon^{-1} \gamma_n )+\cdots+w_n(\kappa^{-1}\epsilon^{-1} \gamma_n )\nonumber\\
=&1+w_1(\det\gamma_{n-1}\oplus \gamma_{n-1})+\cdots+w_n((\det\gamma_{n-1}\oplus \gamma_{n-1})\nonumber\\
=&(1+w_1)(1+w_1+\cdots+w_{n-1})\nonumber\\
=&1+(w_2+w_1^2)+(w_3+w_1w_2)+\cdots+(w_{n-1}+w_1w_{n-2})+w_1w_{n-1}.\label{eq:kapparep}
\end{align}

We will prove that the mapping $t(\kappa^*)\colon H^*(BSO_n)\rightarrow H^*(BO_{n-1})$  preserves weight strictly.
We recall that a filtered mapping 
$\phi\colon H^*(X)\rightarrow H^*(Y)$ strictly preserves 
filtrations if $F^i H^*(X) = \phi^{-1}(F^i H^*(Y))$.

To see first $t(\kappa^*)$ is injective, we recall $\Ker t(\kappa^*)\subset \Ker t(\iota^*)=(w_n)$.
Suppose $0\neq x\in H_{SO_n}^*$ is in $\Ker t(\kappa^*)$.
It is divisible by $w_n$, and there exists $x_j$
such that $x=w_n^jx_j$ and $x_j$ is not divisible by $w_n$.
The image $t(\kappa^*) x= w_{n-1}^j w_1^j \kappa^*x_j$ is zero,
so $t(\kappa^*)x_j$ is zero, which is contradiction.

Let $x$ be an element of $H^*(BSO_n)$ such that
$t(\kappa^*)x$ is contained in the weight $0$ part of $H^*(BO_{n-1})$.
We can write $x=P(w_i)$ with a polynomial of $w_i$.
$t(\kappa^*) x=P(t(\kappa^*) x)$ is a polynomial of $c_i$,
so that each  term will be a square of a monomial of $w_i$.
By \eqref{eq:kapparep} it is possible, by induction on lexicographic order on monomials,
only if $P$ is a square.
We have proved that $t(\kappa^*)$
is strict on the weight $0$ part.

If $y=t(\kappa^*)x$ has weight $k > 0$, we switch to the
motivic cohomology.
Then we can find $\overline{x} \in H_{SO_n}^{*,*}$ 
and $\overline{y}\in H_{O_{n-1}}^{*,*}$ such that 
$\kappa^*\overline{x}=\tau^f\overline{y}$ for some $f\in \bb{Z}_{\geq 0}$.
We have that
\begin{equation*}
\kappa^*Q_{j_0}\ldots Q_{j_{k-1}}\overline{x}=\tau^fQ_{j_0}\ldots Q_{j_{k-1}}\overline{y}.
\end{equation*}
for some $\{j_0,\ldots,j_{k-	1}\}\subset\{0,\ldots,n-2\}$ and
$Q_{j_0}\ldots Q_{j_{k-1}}\overline{y}$ is a non-zero element in $H^{2h,h}$
for it has zero image by all $Q_i$.
It means $f=0$ and $x$ has weight equal to $k$ by the theorem of Yagita.
Finally $t(\kappa^*) c_n= c_1 c_{n-1}$
implies the following.
\begin{prop}\label{prop:kpwt}
The mapping
$t(\kappa^*)\colon H^*_{SO_n}/c_n \rightarrow H^*_{O_{n-1}}/c_1$ 
is injective and it preserves weight strictly.
\end{prop}
\begin{cor}
\begin{enumerate}[(1)]
\item \ For odd $n$, the Stiefel-Whitney class $w_l$ (~$2\leq l\leq n$) in $H^*(BSO_n)$ has  weight $l$ if $l$ is even, and $l-2$ if $l$ is odd.
\item \ For even $n$, the Stiefel-Whitney class $w_n$ has weight $n-2$
in $H^*(BSO_n)$.
For other $w_l$ ($2\leq l\leq n$),
$w_l$ has weight $l$ if $l$ is even, and $l-2$ if $l$ is odd.
\item \ $Q_0w_{2l}=w_{2l+1}$ for all $ 1 \leq l \leq n/2$,
where  we regard $w_{n+1}=0$.
\end{enumerate}
\end{cor}
\begin{proof}
Consider  $t(\kappa^*) w_l$ in $H_{O_{n-1}}^*$.
If $l\neq n$, 
$
t(\kappa^*w_l)=w_l+w_1w_{l-1}=m{[1^l]} + m{[1]}m{[1^{l-1}]}
$
from \eqref{eq:kapparep}.
If $l$ is odd, $t(\kappa^*)w_l=m{[2,1^{l-1}]}$.
If $l$ is even, $
t(\kappa^*)w_l=w_l+w_1w_{l-1}=m{[1^l]} + m{[2, 1^{l-1}]}.
$
For $l=n$, we have that
$
t(\kappa^*)w_n=w_1w_{n-1}=m{[2,1^{n-2}]}.
$
\end{proof}

\subsection{}
The stratified method works fine in ordinary cohomology,
and the result is compatible with the realization:
\begin{equation}
\rightarrow H_{SO_n}^{*-2}\xrightarrow{t({s}_*)}H_{SO_n}^*\xrightarrow{t(t^*)} H_{SO_n}^*(\bb{A}^n\smallsetminus \{0\})\xrightarrow{}.
\end{equation}
We can identify as in \eqref{eq:identify1}
$
H_{SO_n}^*(\bb{A}^n\smallsetminus \{0\})= H_{SO_n}^*/c_n
$.

The sequence  \eqref{eq:locseq2-1} becomes
\begin{equation}\label{eq:topseq2}
\rightarrow H_{SO_{n-2}}^{*-2}\xrightarrow{t(i_*)} H_{SO_{n}}^*/c_{n}
\xrightarrow{t(\kappa^*)} H_{O_{n-1}}/c_1\xrightarrow{t(\delta)}H_{SO_{n-2}}^{*-1}\rightarrow 
\end{equation}

We have the following data on \eqref{eq:topseq2}.
These are direct consequences of Proposition 2.1, \eqref{eq:kapparep},
and $t(\delta)$ is an $H^*_{SO_n}$-module mapping.
We may compare the result of Yagita in \cite[\S 8]{yag2010} by \eqref{eq:locseq2-3}.
Note we will write  the image in the quotient as the same letter, as is $w_i\in H_{SO_{2m}}^*/c_{2m}$.
\begin{prop}
\begin{enumerate}[(1)]
\item \ $t(i_*)=0$. 
\item \ $t(\kappa^*)w_i=w_i + w_{i-1}w_1$ for $2\leq i \leq n-1$, 
and $t(\kappa^*)w_{n}=w_{n-1}w_1$.
\item \ 
$t(\delta)(w_1)=1$ and  $t(\delta)(zw_1)=z$ for $z$ is a polynomial of $w_i$ for $2\leq i\leq n-2$.
\end{enumerate}
\end{prop}
\subsection{$H^{*,*}(BSO_{2m+1})$}
Since $H^*_{O_n}$ and $H^*_{SO_n}$ are generated by the Stiefel-Whitney classes,
the realization mapping 
$t^{*,*}$ are epimorphisms in all these cases.
As in 1.3, $t \colon H^{p,q}_{O_n}
\rightarrow H^p_{O_n}$ is injective for every pair $(p,q)$.

When $n=2m+1$, we have that $O_{2m-1}\cong SO_{2m-1}\times \mu_2$.
This direct product structure induces the mapping 
$\pi \colon BO_{2m-1}\rightarrow BSO_{2m-1}$ deduced from multiplication of 
determinant $g\mapsto \det(g)g$ if $g\in O_{2m-1}$.
Then $\pi \cdot \epsilon = 1$ and $( \epsilon \cdot \pi)( \epsilon \cdot \pi)=( \epsilon \cdot \pi)$.
\begin{equation}\label{eq:oddB}
H^{*,*}_{O_{2m-1}}/c_1
= H^{*,*}_{\Z/2 \times SO_{2m-1}}(\mathbb{G}_m) =
H^{*,*}_{SO_{2m-1}}\{1, x\}
\end{equation}
where $x=\overline{w_1}$ has degree (1,1).
It implies the realization $t : H^{p,q}_{SO_{2m-1}} \to H^{p}_{SO_{2m-1}}$
is injective.
We can compare 
\eqref{eq:locseq2-1} and \eqref{eq:topseq2} through 
monomorphisms $t$.
Because $t(\kappa^*) t_{SO_{2m+1}}= t_{O_{2m}} \kappa^*$
is injective, $\kappa^*$ is injective, then 
we get
\begin{equation}\label{odd}
\xymatrix@C-15pt{
0 \ar[r] & H^{\ast,\ast}_{SO_{2m+1}}/c_{2m+1} \ar[r]^-{\kappa^*}&
H^{\ast,\ast}_{O_{2m}}/c_1 \ar[r]^-{\delta_2}  & 
H^{\ast-1, \ast-1}_{SO_{2m-1}} \ar[r] & 0
}
\end{equation}

\section{$H^{*,*}(BSO_{2m})$}
\subsection{}
Let us denote by $Y_{m}$ the kernel of 
$H^{*,*}_{SO_{2m}} \to H^{*}_{SO_{2m}}$.
We will inductively determine $Y_m$ so that
we assume $Y_m$ is known.
We denote by 
$\overline{Y_{m+1}}$
the kernel of
$t : H^{*,*}_{SO_{2m+2}}/c_{2m+2} \to H^{*}_{SO_{2m+2}}/c_{2m+2} $.
Later we will see in 3.3 that $(\Ker t_{SO_{2m}})/c_{2m+2}= \overline{Y_{m+1}}$.
We can combine \eqref{eq:locseq2-1} and \eqref{eq:topseq2}
to 
the following commutative diagram \label{eq:locseq3}
\begin{equation}
\xymatrix{
\ar[r]^-{\delta} & H_{SO_{2m}}^{*-2,*-1} 
\ar[r]^-{{i}_*}\ar[d]_{t_{2m}} & H_{SO_{2m+2}}^{*,*}/c_{2m+2} 
\ar[r]^-{\kappa^*}\ar[d]_{t_{2m+2}} & H_{O_{2m+1}}^{*,*}/c_1
\ar[r]^-{\delta}\ar[d]^{t_{2m+1}} & H_{SO_{2m}}^{*-1,*-1}\ar[r] \ar[d]&\\
\ar[r] & H_{SO_{2m}}^{*-2} \ar[r]^-{0} & 
H_{SO_{2m+2}}^*/c_{2m+2}\ar[r]^-{t(\kappa^*)} & H^*_{O_{2m+1}}/c_1\ar[r]^-{t(\delta)}&
H_{SO_{2m}}^{*-1}
\ar[r]^-{0}& 
.
}
\end{equation}
$t_{2m+1}$ on the right column is injective for each degree.
$\overline{w_1}$ has degree $(1,1)$ and $t\overline{w_1}=w_1$.

We obtain that
$\Im({i}_*)= \Ker(t_{2m+2})
\subset H_{SO_{2m+2}}^{*,*}/c_{2m+2}$,
using diagram-tracing argument.  Note  $t_{2m+1}$ is injective.
The restriction of $i_*$ gives an isomorphism between
$\Coker \delta$ and $\overline{Y_{m+1}}$,
and it takes $Y_m$ into $\overline{Y_{m+1}}$.
We can see $\Im \delta \cap Y_m = 0$.
For if there is $u = \delta v$ such that $t(u)=0$,
there exists $w$ such that $t(\kappa^*) t(w)= t(v)$ 
.
Because $t(\kappa^*)$  strictly preserves filtrations,
there is $w' \in H^{*,*}_{SO_{2m+2}}/c_{2m+2}$
such that $\kappa^* w'= v$. 
This means  $u=0$.

The restriction of $i_*$ gives an acyclic subcomplex $Y_m \to Y_m$, where
the second copy is in $\overline{Y_{m+1}}$, in \eqref{eq:locseq3},
and we get 
\begin{equation}\label{eq:sq3}
\xymatrix@C-5pt{
 \ar[r] & \left(H^{\ast,\ast}_{SO_{2m+2}}/Y_{m}\right)/c_{2m+2}
\ar[r]^-{\kappa^*}&
H^{\ast,\ast}_{O_{2m+1}}/c_1 \ar[r]^-{\delta_1}& 
H^{\ast-1,\ast-1}_{SO_{2m}}/Y_{m} \ar[r] & }
\end{equation}
where the induced $t_{2m}$ is now a monomorphism
and $\Coker \delta_1 \cong \overline{Y_{m+1}}/Y_m$.

By Proposition 2.3 we know that $\Coker \delta_1$ is isomorphic to
the cokernel of the restriction
$\iota^*_{2m+1} : H^{\ast,\ast}_{SO_{2m+1}} \to H^{\ast,\ast}_{SO_{2m}}/Y_m$,
up to permutation of coordinates,
which makes
a part of  the following commutative exact sequences
\begin{equation}\label{eq:final}
\xymatrix@C-15pt{
0 \ar[r] & H^{\ast,\ast}_{SO_{2m+1}}/c_{2m+1} \ar[d]^{\iota_{2m+1}^*}\ar[r]^-{\kappa^*}&
H^{\ast,\ast}_{O_{2m}}/c_1 \ar[r]^-{\delta_2} \ar[d]^{\iota_{2m}^*} & 
H^{\ast-1, \ast-1}_{SO_{2m-1}} 
\ar[d]^{\iota_{2m-1}^*}\ar[r] & 0\\
0 \ar[r] & (H^{\ast,\ast}_{SO_{2m}}/Y_{m})/c_{2m}   \ar[r]^-{\kappa^*}&
H^{\ast,\ast}_{O_{2m-1}}/c_1 \ar[r]^-{\delta_1}& H^{\ast-1,\ast-1}_{SO_{2m-2}}/Y_{m-1}\ ,
  & }
\end{equation}
where $\iota^*_{2m}$ is an epimorphism.

\subsection{}
Since 
$\kappa^* \overline{w_{2m}}= 
\overline{w_{2m}}+ \tau \overline{w_{1}w_{2m-1}}$
in $H^{*,*}_{O_{2m}}$,
$\overline{w_{2m}}$ has weight $2m$ in $SO_{2m+1}$.
On the other hand  $\overline{w_{2m}}$ in $H^{*,*}_{SO_{2m}}$ satisfies
$\kappa^* \overline{w_{2m}}= \overline{w_{1}w_{2m-1}}$,
which implies it has weight $2m-2$.
That is,
$\iota_{2m+1}^*\overline{w_{2m}}=\tau \overline{w_{2m}}$,
and  $\overline{w_{2m}}$ presents a nonzero element in $\Coker \iota_{2m+1}^*$,
where $\tau \overline{w_{2m}}=0$.

We want to claim that  $\overline{w_{2m}}$
is a generator of
$\Coker \iota_{2m+1}^*$ over the ring of even Chern classes.
Since $\kappa$ and $\iota$ commute,
we have to consider elements having factor $\overline{w_{2m}}$ only.

Consider the long exact sequence associated with \eqref{eq:final}: 
\begin{equation}
\Ker \iota^*_{2m+1} \to \Ker \iota^*_{2m} \to \Ker \iota^*_{2m-1} 
\overset{\partial} \to \Coker \iota^*_{2m+1} \to 0.
\end{equation}

Connecting mapping $\partial$ is $(\kappa^*)^{-1} \cdot \iota^*_{2m} \cdot (\delta_2)^{-1}$.
Let us take $z$ in $\Ker \iota^*_{2m-1}$.
We can take $(\delta_2)^{-1} z =\overline{w_1z}$.
If it is in ordinary cohomology, $\Coker \iota^*_{2m+1}=0$, for 
we can take $w_{2m}$ as $(\delta_2)^{-1} w_{2m-1}$ which 
go to zero by $\iota^*_{2m}$, but in motivic cohomology
$\overline{w_{2m}}$ has strictly larger weight than $\overline{w_1w_{2m-1}}$, hence 
$\overline{w_{2m}}$ presents a non-zero element in $\Coker \iota^*_{2m+1}$ as above.

Therefore it suffices to prove that in $H_{O_{2m}}^{*,*}$,
$z w_{2m}$ has weight less than or equal to the weight of $z w_1 w_{2m-1}$,
if $z$ is not a polynomial of even Chern classes.
Let $z$ be an element of monomial base $m[i_{\cdot}]$.
First we assume not all $i$'s are even.
We will use induction of length of $z$.
If it has length $1$, $m[2i+1]m[2,1^{2m-2},0]$ and $m[2i+1]m[1^{2m}]$
have weight $2m-1$ ($i>0$).
$m[2i+1, 2j+1]m[2,1^{2m-2},0]$ have weight more than $2m-2$ and $m[2i+1,2j+1]m[1^{2m}]$
have weight $2m-2$.
If in a monomial of length $j \geq 2$ all $i$ are odd, $zw_{2m}$ have weight $2m-j$,
and $zw_1w_{2m-1}$ also  have weight $2m-j$.
In general let $2j$ be the largest even member of $i_{\cdot}$ which has multiplicity $k$.
Let $z'$ be the partition discarding $2j$'s from $z$. 
Then $z$ is equal to $z'(w_k)^{2j}$ plus elements of length less than that of $z$.
$z'(w_k)^{2j}w_{2m}$ has the same weight as that of 
$z'w_{2m}$, and similarly for $w_1w_{2m-1}$.
It finishes the inductive step.

If all $i$'s are even, it is a polynomial of $(w_i)^2$ and
$zw_{2m}$ has weight $2m$ but $zw_1w_{2m-1}$ has weight $2m-2$.
If $z$ contains $w_{2j-1}^2$, there exists u such that $zw_{2m} + u w_1w_{2m}$ 
has weight $2m-2$, hence it does not represent nonzero elements in 
$\Coker \iota^*_{2m+1}$.

\subsection{}\label{subsec:cn}
We know that $Y_1=0$, and the cokernel of
$\delta$ is easily seen to be the image of $\Z/2[c_2]\{\overline{w_2}\}$,
here $\overline{w_2} \in H^{2,1}_{SO_{2}}$ in $H^{\ast,\ast}_{SO_{4}}/c_4$.
Let us define $y_{0,2}={i_2}_* \overline{w_2}$.
Then $Y_2=\Z/2[c_2, c_4]\{y_{0,2}\}$~\cite[\S 8]{yag2010}.

Assume we know that $Y_{m}$ is generated by $y_{i, m}$ 
over $\Z/2\,[\,c_2, \ldots ,c_{2m}]$ for $0\leq i \leq m-2$.
$i_*(y_{i,m})$ are non-zero elements in
$H_{SO_{2m+2}}^{*,*}/c_{2m+2}$ ($i_*$ is that in \eqref{eq:locseq3}.)
For $i_*y_{i,m}$ have degree  $2m+2$ there exist 
$y_{i,m+1}$ in $H_{SO_{2m+2}}^{*,*}$.
We will see that $(Y_{m+1}/c_{2m+2})/i_*(Y_{m})$
is generated by $i_* \overline{w_{2m}}$
over $\Z/2\,[\,c_2, \ldots ,c_{2m}]$.
There exists a unique element $y_{m-1,m+1}$ in $H^{2m+2,2m}(BSO_{2m+2})$
whose image by quotient of $c_{2m+2}$ is $i_* \overline{w_{2m}}$.
We get an isomorphism 
$$Y_{m+1} = \Z/2\,[\,c_2, c_4, \ldots ,c_{2m+2}]
\{y_{0,m+1},\ldots , y_{m-1,m+1}\}.$$
Multiplicative relations are consequence of 
the projection formula \cite{soulynin}.

Finally
we consider ${s}_*$ in (\ref{eq:locseq1})
as an $H_{GL_{2m}}^{*,*}$-module mapping.
$c_{2m} \cup \ $ is injective on 
$\Ker t^{*,*}_{SO_{2m}}$ and $H_{SO_{2m}}^*$, so that
we can identify
$H_{SO_n}^{*,*}(\bb{A}^n -\{0\}) = H_{SO_n}^{\ast,\ast}/c_n$
for all $n$.

\bibliographystyle{plain}

\end{document}